\documentclass[11pt]{amsart}

\usepackage{epigamath-voisin}

\usepackage[english]{babel}

\numberwithin{equation}{section}

\usepackage[shortlabels]{enumitem}
\setlist[enumerate,1]{label={\rm(\roman*)}, ref={\rm\roman*}} 
\setlist[enumerate,2]{label={\rm(\alph*)}, ref={\rm\alph*}} 

\usepackage{amssymb,amsmath,amscd,graphicx, enumitem,latexsym,amsthm,tikz-cd}
\usepackage{amssymb,latexsym,amsmath,amscd,graphicx, amsxtra,verbatim}
  \usepackage[all]{xy}
  \usepackage[mathscr]{euscript}

\theoremstyle{plain}
\newtheorem{theorem}{Theorem}

\newtheorem{proposition}[theorem]{Proposition}

\newtheorem{lemma}[theorem]{Lemma}

\newtheorem{proposition.definition}[theorem]{Proposition/Definition}

\newtheorem*{claim}{Claim}

\newtheorem{theoremalpha}{Theorem}

\newtheorem{propositionalpha}[theoremalpha]{Proposition}

\newtheorem*{lemmann}{Lemma}

\theoremstyle{definition}

\theoremstyle{remark}

\newtheorem{remark}[theorem]{Remark}
\newtheorem*{question}{Question}

\newcounter{defcounter}
\newcommand{\lra}{\longrightarrow}

\newcommand{\PP}{\mathbf{P}}

\newcommand{\ZZ}{\mathbf{Z}}

\newcommand{\QQ}{\mathbf{Q}}

\newcommand{\pr}{\prime}

\newcommand{\dra}{\dashrightarrow}

\newcommand{\Linser}[1]{| \mspace{1.5mu} {#1}
\mspace{1.5mu} |}
\newcommand{\linser}[1]{\Linser{  {#1}  }}

\DeclareMathOperator{\autocorr}{autocorr}
\DeclareMathOperator{\gon}{gon}
\DeclareMathOperator{\pro}{pr}

\newcommand{\ol}[1]{\overline{#1}}
\DeclareMathOperator{\Pic}{Pic}
\DeclareMathOperator{\irr}{irr}

\DeclareMathOperator{\Trace}{Tr}

\DeclareMathOperator{\End}{End}
\DeclareMathOperator{\Id}{Id}
\DeclareMathOperator{\Hom}{Hom}
\DeclareMathOperator{\defi}{def}
\DeclareMathOperator{\Hdg}{Hdg}
\DeclareMathOperator{\GL}{GL}
\DeclareMathOperator{\HS}{HS}

\def\shortbar{%
\smash{\scalebox{0.4}[1.0]{$-$}}}
\makeatother
      
\makeatletter
\newcommand{\xdashrightarrow}[2][]{\ext@arrow 0359\rightarrowfill@@{#1}{#2}}
\def\rightarrowfill@@{\arrowfill@@\relax\shortbar\dashrightarrow}
\def\arrowfill@@#1#2#3#4{%
  $\m@th\thickmuskip0mu\medmuskip\thickmuskip\thinmuskip\thickmuskip
   \relax#4#1
   \xleaders\hbox{$#4#2$}\hfill
   #3$%
}
\makeatother

\newcommand{\longdashrightarrow}{\xdashrightarrow{\hspace{9pt}}}

\YearArticle{2023} \EpigaArticleNr{10} \ReceivedOn{March 3, 2023}
\InFinalFormOn{July 17, 2023}
\AcceptedOn{August 7, 2023}

\title{Measures of association between algebraic varieties, II: Self-correspondences}
\titlemark{Measures of association between algebraic varieties, II}

\author{Robert Lazarsfeld}
\address{Department of Mathematics, Stony Brook University, Stony Brook, New York 11794, USA}
\email{robert.lazarsfeld@stonybrook.edu}
\author{Olivier Martin}
\address{Department of Mathematics, Stony Brook University, Stony Brook, New York 11794, USA\\
  \indent Instituto de Matemática Pura e Aplicada, Rio de Janeiro, RJ 22460-320, Brazil}
\email{olivier.martin@stonybrook.edu}

\authormark{R.~Lazarsfeld and O.~Martin}

\AbstractInEnglish{}



\acknowledgement{Research  of the first author was partially supported by the Simons Foundation.}

\dedication{Dedicated to Claire Voisin on the occasion of her sixtieth birthday}

\begin{document}


\maketitle

\begin{prelims}



\DisplayKeyWords

\medskip

\DisplayMSCclass

\end{prelims}


\newpage

\setcounter{tocdepth}{1}

\tableofcontents


\section{Introduction}\label{intro}

In our previous paper \cite{MOA}, we introduced and studied some invariants intended to measure how far from birationally isomorphic two given varieties $X$ and $Y$ of the same dimension might be. These were defined by studying the minimal birational complexity of correspondences between $X$ and $Y$. Following  a suggestion of Jordan Ellenberg, the present note continues this line of thought by investigating self-correspondences of a given variety.

Let $X$ be a smooth complex projective variety of dimension $n$. By an auto-correspondence of $X$, we understand a smooth projective variety $Z$ of dimension $n$ sitting in a diagram
\begin{equation}
  \begin{gathered} \label{Diagram1.Intro}
\xymatrix{
& Z \ar[dl]_a\ar[dr]^b  \\
X & & X
 }
\end{gathered}
\end{equation}
with $a$ and $b$ dominant and hence generically finite. 
We assume that $Z$ maps birationally to its image in $X \times X$ (so that general fibers of $a$ and $b$ are identified with subsets of $X$). The \textit{auto-correspondence degree} of $X$ is defined to be
\begin{equation} \label{Def.Autocorrdeg.Intro}
\autocorr(X) \ = \ \min_{Z \ne \Delta} \,\left\{ \deg(a) \cdot \deg(b)   \right\},
\end{equation}
the minimum being taken over all such $Z$ excluding those that map to the diagonal. Thus
$\autocorr(X) = 1$ if and only if $X$ admits non-trivial birational automorphisms. By considering the fiber square of a rational covering $X \dra \PP^n$, one sees that  
\[ \autocorr(X) \ \le \ \left( \irr(X) - 1 \right)^2, \]
where the degree of irrationality $\irr(X)$ is defined to be the least degree of such a covering (see Section~\ref{sec1},  below). Our intuition is that equality holding means that $X$ is ``as far as possible" from having any interesting self-correspondences of low degree. 

Our main results are as follows.

\begin{propositionalpha} \label{NewPropositionA}
If\, $X$ is a very general curve of genus $g \ge 3$, then
\[ \autocorr(X) \ = \ \left( \gon(X) -1 \right)^2, \]
and minimal correspondences arise from the fiber square of a gonal map.
\end{propositionalpha}

\begin{theoremalpha} \label{NewTheoremB}
Let $X \subseteq \PP^{n+1}$ be a very general hypersurface of degree $d \ge 2n + 2$. Then
\[ 
\autocorr(X) \ = \ (d-2)^2 \ = \ \left( \irr(X) -1\right)^2,
\]
and again minimal correspondences are birational to the fiber square of projection from a point.
\end{theoremalpha}

In fact, we classify all self-correspondences in a slightly wider numerical range: see Theorem~\ref{NewTheoremD'}.

If $X$ is a hyperelliptic curve of genus $g$, then $\autocorr(X) = 1$ since $X$ has a non-trivial automorphism whose graph is a non-diagonal copy of $X$ sitting in $X \times X$.   David Rhyd asked whether there are any unexpected hyperelliptic curves in this product. Our final result asserts that there are not.

\begin{theoremalpha}   \label{NewTheoremC}
Let $X$ be a very general hyperelliptic curve of genus $g\geq 2$. The only hyperelliptic curves in $X\times X$ are
\begin{itemize}
\item the fibers of the projection maps, 
\item the diagonal, and
\item the graph of the hyperelliptic involution.
\end{itemize}
In particular, the image of any hyperelliptic curve in $X\times X$ under the Abel--Jacobi map is geometrically degenerate in $J(X)\times J(X)$; \textit{i.e.}, it generates a proper subtorus of that product.
\end{theoremalpha}

By a hyperelliptic curve in $X \times X$, we mean an irreducible curve $Z \subseteq X \times X$ whose normalization is hyperelliptic.

As pointed out by the referee, Theorem~\ref{NewTheoremC} is closely related to work of Schoen from \cite{Schoen}. In particular, \cite[Proposition 4.1(2)]{Schoen} implies Theorem~\ref{NewTheoremC} for genus $2$ curves. Schoen also gives examples of hyperelliptic curves $X$ such that $X\times X$ contains finitely many hyperelliptic curves (see \cite[Lemma 1.5 and Proposition 2.2]{Schoen}).

A consequence of the proof of Theorem~\ref{NewTheoremC} (see Proposition~\ref{PropositionE}) is that given a very general hyperelliptic curve $X$ and any hyperelliptic curve $C$ whose Jacobian dominates $J(X)$, we have
$$\Hom(J(C),J(X))\cong \mathbb{Z}.$$
This rigidity statement complements recent results of Naranjo and Pirola concerning dominant morphisms from hyperelliptic Jacobians (see \cite[Theorems~1.4 and~1.6]{NaranjoPirola}).

We work throughout over the complex numbers.

\subsection*{Acknowledgments}
We are grateful to Jordan Ellenberg and David Rhyd for raising the
questions addressed here. We also thank Mark Green and Radu Laza for
showing us proofs of the lemma in Section~\ref{sec3} before we were
able to find a reference in the literature.

It is an honor to dedicate this paper to Claire Voisin on the occasion
of her sixtieth birthday. Her influence on both the field as a whole
and the work of the two authors has been immense.

\section{Preliminaries and proof of Proposition~\ref{NewPropositionA}}\label{sec1}

We start with some general remarks about the auto-correspondence degree. Given a smooth complex projective variety $X$ of dimension $n$, its auto-correspondence degree $\autocorr(X)$ is defined as in the introduction. Evidently, this is a birational invariant of $X$. 

Note that if $X$ admits a rational covering $X \dra \PP^n$ of degree $\delta$, then
\[ \autocorr({X}) \ \le \ (\delta -1)^2.  \tag{$\ast$} \]
In fact, replacing $X$ with a suitable birational model, we can suppose that $X \to \PP^n$ is an actual morphism. Then 
\[ X \times_{\PP^n} X \, \subseteq \, X \times X \]
contains the diagonal $\Delta_X$ as an irreducible component. The union of the remaining components $Z^\pr \subseteq X \times_{\PP^n} X$ has degree $\delta - 1$ over each of the factors, and ($\ast$) follows. In particular, 
\begin{equation} \label{autocorr.bound.eqn}
\autocorr(X) \, \le \,   ( \irr(X) - 1 )^2,
\end{equation}
where $\irr(X)$ denotes the minimal degree of such a rational covering $X \dra \PP^n$. Our main results assert that in several circumstances equality holds in \eqref{autocorr.bound.eqn} and that the minimal correspondences arise as just described. We will say in this case that $Z$ is residual to the fiber square of a minimal covering of $\PP^n$. 

As in the earlier works \cite{BCD,BDELU,MOA}, the action of a correspondence on cohomology plays a central role. In the situation of  diagram \eqref{Diagram1.Intro}, $Z$ gives rise to endomorphisms
\[
Z_* \, = \, b_* \circ a^*, \ \ Z^* \, = \, a_* \circ b^*    \]
of the Hodge structure $H^n(X)$. We denote by
\[     Z_*^{n,0} \ = \ \Trace_b \circ a^*, \ \   {Z^*}^{n,0} \ = \ \Trace_a \circ b^*\]
the corresponding endomorphisms of the space $H^{n,0}(X)$ of holomorphic $n$-forms on $X$. In the cases of interest, these will act as a multiple of the identity, allowing us to use a variant of the  arguments from the cited papers involving Cayley--Bacharach.

We now turn to the proof of Proposition~\ref{NewPropositionA}. We then suppose that $X$ is a very general curve of genus $g \ge 3$ and that $Z \rightarrow X \times X$ is a correspondence as in \eqref{Diagram1.Intro} that computes the auto-correspondence degree of $X$. The generality hypothesis on $X$ implies first of all that 
\[
\Pic(X \times X) \ = \ a^* \Pic(X) \, \oplus \, b^* \Pic(X) \, \oplus \, \ZZ \cdot \Delta, 
\]
and hence the image of $Z$ in $X \times X$ is defined by a section of
\begin{equation} \label{Defining.Eqn.Z.for.Curves}
\left( B \, \boxtimes \, A \right) (-m\Delta) 
\end{equation}
for some line bundles $A, B$ on $X$ and some $m \in \ZZ$. Note that then
\[ \deg(a) \, = \, \deg(A) - m, \ \ \deg(b) \, = \, \deg(B) - m. \]
Moreover, both maps
\[
Z_*^{1,0}, \ {Z^*}^{1,0} \colon H^{1,0}(X) \lra H^{1,0}(X) 
\]
are multiplication by $-m$. 

We start by proving that
\[ \deg(a),\,  \deg(b) \ \ge \ \gon(X) -1, \]
which will imply that $\autocorr(X) = (\gon(X) -1)^2$.  We may suppose that  $m \ne 0$, for if $m = 0$, then we are in the setting of \cite[Example 1.7]{MOA} and $\deg(a), \deg(b) \ge \gon(X)$.  Now  fix a general point $y \in X$, and suppose that
\[   b^{-1}(y) \ = \ x_1 + \ldots + x_\delta, \]
where $\delta = \deg(b)$. Then for any $\omega \in H^{1,0}(X)$, we have
\[    \omega(x_1) + \ldots +  \omega(x_\delta)\, = \, Z_*^{1,0}(\omega)(y) \, = \, -m \cdot \omega(y). \]
It follows that the $\delta + 1$ points $y, x_1, \ldots , x_\delta$ do not impose independent conditions on $H^{1,0}(X)$, and hence they move in at least a pencil. In other words, $\deg(b) + 1 \ge \gon(X)$ and similarly $\deg(a) + 1 \ge \gon(X)$, as required.

Assuming that $\deg(a) = \deg(b) = \gon(X) - 1$, it remains to show that a minimal correspondence arises from the fiber square of a pencil. For this, we first of all rule out the possibility that $m < 0$. In fact, by intersecting $Z$ with the diagonal, one finds that
$\deg(A) \ + \ \deg(B) \ge -m\cdot (2g - 2)$,
and hence $\deg(a) + \deg(b) \ge -m \cdot(2g)$. But unless $g = 3$,  this is impossible if $m < 0$ since $2 \cdot (\gon(X) -1)  \le g+1$. When $g = 3$, one needs to rule out the existence of line bundles $A, B$ of degree $2$ such that $r\left(A(y)\right) = r\left(B(y)\right) \ge 1$ for every $y \in X$, and this follows from the well-known description of pencils of degree $3$ on a smooth plane quartic. (See also Remark~\ref{Tang.Correspondence.Remark}.)

Returning to the setting of \eqref{Defining.Eqn.Z.for.Curves}, now assume  that $m > 0$. Then
for every $x \in X$, 
\[    a^{-1}(x) \, \in \, \linser{ A(-m\cdot x) }, \ \  b^{-1}(x) \, \in \, \linser{ B(-m\cdot x) }, \]
which implies that 
\begin{equation} \label{Eqn.r(A).r(B)}  r(A),\, r(B) \ \ge \ m. \end{equation}
We will use this to show that $\deg(a)$ and $\deg(b)$ are minimized when $A$ and $B$ move in pencils. 

 In fact, write $d = \deg(A)$. If $A$ is non-special, then $r(A) = d-g$, so $\deg(a) = d -m \ge g$ thanks to \eqref{Eqn.r(A).r(B)}, and we get a map of smaller degree from a gonal pencil.\footnote{The case $g = 3$ requires a special argument here that we leave to the reader.} Therefore assume  that $A$ is special, so that
 \[   A \, \in \, W^m_d(X). \]
 We may suppose that $X$ is Brill--Noether general, in which case
 \[\rho(m,d,g) \ = \ g-(m+1)(g-d+m)\ \geq \ 0.\]
It follows that
\[(m+1)d \ \geq \  mg+m(m+1)\]
and
\[d \ \geq \  \left(\frac{m}{m+1}\right)g+m,\]
so that
\[\deg(a) \, = \, d-m \ \geq \   \left(\frac{m}{m+1}\right)g.\]
This is minimized when $m = 1$ and similarly for $B$. Thus we can assume that $r(A) = r(B) =1$ and that the image of $Z$ lies in the linear series
\[   \linser{  \left( B \, \boxtimes \, A \right) (-\Delta) } \]
on $X \times X$. But this series is empty unless $A =B$, in which case it consists exactly of the residual to the diagonal in the fiber square of the pencil defined by $A$. This completes the proof.

\begin{remark} \label{Tang.Correspondence.Remark}
Suppose that $X \subseteq \PP^2$ is  a smooth plane curve of degree $d>3$; then the correspondence defined  as the closure of 
\[\left\{(x,y)\in X\times X\mid y\neq x\text{ is in the embedded tangent line to } X \text{ at }x\right\}\]
dominates the first factor with degree $d-2$ but fails to arise from the fiber square of projection from a point on $X$. The degree of the second projection is $d(d-1)$, which is much greater than $d-1 = \gon(X)$. 
\end{remark}

\begin{remark} As the referee remarks, the preceding result leads to two interesting questions to which we don't know the answers. First, if $X$ is a very general $k$-gonal curve, is it true that
\[
\autocorr(X) \, = \, (k-1)^2?
\]
We suspect that this should be the case. Second, does Proposition~\ref{NewPropositionA} hold if we replace ``very general'' with ``general''? 
\end{remark}

\section{Proof of Theorem~\ref{NewTheoremB}}\label{sec2}

In this section we prove the following refinements of Theorem~\ref{NewTheoremB} from the introduction. 
\begin{theorem}   \label{NewTheoremD}
Let 
$X \subseteq \PP^{n+1}$ 
be a very general hypersurface of degree  $d \ge 2n + 2$, and consider a self-correspondence $Z$ as in diagram \eqref{Diagram1.Intro}:
\[
\xymatrix{
& Z \ar[dl]_a\ar[dr]^b  \\
X & & X.
 } \]
 Assume that $Z$ does not map to the diagonal. Then
\[ \deg(a) \, \ge \, d-2, \ \ \deg(b) \, \ge \, d-2, \]
and hence $\autocorr(X) = (d-2)^2$. 
\end{theorem}

\begin{theorem} \label{NewTheoremD'}
In the situation of Theorem~\ref{NewTheoremD}, assume in addition that $\deg(a)\leq 2d-2n-3$.
\begin{enumerate}
\item\label{D'i} If\, $\deg(b) \le d-2$, then $\deg(a) = d-2$ and $Z$ is birationally residual to the fiber square of projection from a point $x_0 \in X$. 
\item\label{D'ii} If\, $\deg(b) = d-1$, then either
\begin{enumerate}
\item\label{D'iia} $Z$ is birational to the fiber product of two rational mappings $\phi_1 , \phi_2 \colon X \dra \PP^{n}$, or 
\item\label{D'iib}  there exist an $n$-fold $Y\!$ and a dominant rational mapping $\phi\colon X \dashrightarrow Y$ of degree $d$ such that $Z$ is birationally residual to the diagonal in the fiber product $X \times{_Y} X$.
 \end{enumerate}
\end{enumerate}
 \end{theorem}
 
 \begin{remark}
   The various possibilities in Theorem~\ref{NewTheoremD'} actually occur. For example, in \eqref{D'iia} one considers projection from two different points in $X$, while \eqref{D'iib}  arises for  a general projection $X \to \PP^n$ from a point off $X$.
 \end{remark}

Turning to the proofs, the arguments follow the  line of attack of \cite{BCD,BDELU,MOA}, so we will  be relatively brief. Fix $X$ and $Z$ as above, and write
\[\delta_a\ =_{\defi}\ \deg (a)\quad\text{and}\quad \delta_b\ =_{\text{def}}\ \deg (b).\]
The first point to observe is that  we may---and do---assume that the endomorphism ring of the Hodge structure $H^n_{\pro}(X,\ZZ)$ is $\ZZ$.

\begin{lemma}\label{lem6}
 If\, $X$ is a very general hypersurface in $\PP^{n+1}$, then 
 $$\End(H^n_{\pro}(X,\ZZ))=\mathbb{Z}\cdot \Id.$$
 Equivalently, $\Hdg^{n,n}(X\times X)$  is generated by the classes of the diagonal and the products $ h_1^ih_2^{n-i}$ $(1 \le i \le n)$, where $h_j=\pro_{j}^* c_1(\mathcal{O}(1)|_X)$.
 \end{lemma}
\begin{proof}The lemma follows from the computation of the algebraic monodromy group for the corresponding variations of Hodge structures (see \cite{Beau} or \cite[Section~10.3]{PS}). In the cases $d=1,2$, the primitive cohomology has rank $0$ and $1$, respectively, so the statement holds trivially. For larger $d$, an element of $\GL(H^n_{\pro}(X,\ZZ))$ is a morphism of Hodge structures if and only if it commutes with the orthogonal group ($n$ even) or the symplectic group ($n$ odd). The centralizer of both of these subgroups is $\mathbb{Z}\cdot \Id$. Alternative arguments were shown to us by Radu Laza and Mark Green.
\end{proof}
It follows from Lemma~\ref{lem6} that 
\[ Z_*^{n,0} \colon H^{n,0}(X) \lra H^{n,0}(X)\]
is multiplication by some integer $c$. Note that then ${Z^*}^{n,0}\colon H^{n,0}(X) \to H^{n,0}(X)$ is multiplication by the same integer $c$. In fact, abusively writing $[Z]$ for the class of the image of $Z$ in $H^*(X \times X)$, one has
$$
[Z]\, \in \, \Hdg^{n,n}(X\times X)\ =\ \left\langle \Delta\, , \,  h_1^ih_2^{n-i}\, | \, 1\leq i\leq n\right\rangle_{\mathbb{Q}},
$$
and of these  classes, only $\Delta$ gives rise to a non-zero map
\[H^{n,0}(X)\longrightarrow H^{n,0}(X)\]
under the identification
\[\Hdg^{n,n}(X\times X)\, \cong \, \End_{\QQ-\HS}(H^n(X)).\]
Moreover, 
\[\Delta_*=\Delta^*=\Id_{H^{n,0}(X)}.\]
As in the previous section, we will need to distinguish between the cases $c=0$ and $c\neq 0$.\\

We start by showing that
\[   \delta_a, \,   \delta_b  \ \ge \ d-2,\]
by an argument parallel to that appearing in Section~\ref{sec1}. 
First, observe that
\begin{equation} \delta_a
\, ,\,\delta_b  \ \ge \ \begin{cases} d-n & \text{ if } c=0,\\
d-n-1 &\text{ if } c\neq 0. \end{cases}\label{lowerbounddelta}\end{equation}
Indeed, if $c=0$, then $Z$ is a traceless correspondence, so given general $x,y\in X$,  the sets $a^{-1}(x)$ and $b^{-1}(y)$ both satisfy the Cayley--Bacharach condition with respect to $H^{n,0}(X)$. Similarly, when $c \ne 0$,  the cycle $Z-c\Delta$ is a traceless correspondence, and hence for general $x,y\in X$, the sets $a^{-1}(x)\cup \{x\}$ and $b^{-1}(y)\cup \{x\}$ also both satisfy the Cayley--Bacharach condition. Inequality \eqref{lowerbounddelta} then follows from \cite[Theorem 2.4]{BCD}.

We next assume  that $\delta_b \le d-1$, aiming for a contradiction when $\delta_b \le d-3$. Fix a general point $y \in X$. The fiber of $Z$ over $y$  sits naturally as a subset of $X$ and hence also $\PP^{n+1}$:
\[ Z_y \, =\, \{x_1,\cdots, x_{\delta_b}\}\, =_{\defi} \, b^{-1}(y)\, \subseteq \, X\, \subseteq \, \PP^{n+1}.\]
Note that if $y$ is general, then the points $x_j$ are distinct. Since $\delta_b+1 \le 2d - 2n + 1$, it follows from \cite[Theorem 2.5]{BCD} and the vanishing of $(Z-c\cdot\Delta)_*$ that
 the finite set $Z_y$ spans a line $\, \ell_y \, \subseteq \,\PP^{n+1}$. In  a similar fashion, the generic fiber $a^{-1}(x)$ spans a line $_x \ell $. Furthermore, if $c\neq 0$, the point $y$ lies on $\ell_y$ and  $x$ lies on $_x \ell$.

Write
 \[X\cdot \ell_y\ =\ \sum_{i=1}^r a_i z_i \, ;\]
we denote by $m(z)$ the multiplicity of $z$ in  $X\cdot \ell_y$, and we note that  the $x_j$ appear among these points. Observe that $m(x_j)$ does not depend on $j$. Indeed, if $m(x_j)$ were to vary, picking out the $x_j$ with the highest multiplicity for each $y\in X$ would define a non-trivial multisection of the generically finite map 
\[b\colon Z\longrightarrow X,\]
thereby violating the irreducibility of $Z$. Moreover, $m(x_j)=1$ for every $j$. Indeed, if $c=0$, then $\delta_b\geq d-n$ and thus $2\delta_b>d$. Since $\sum m(x_j)\leq d$, we see that $m(x_j)=1$. If $c\neq 0$, then $\delta_b\geq d-n-1$ and $2\delta_b+1>d$. Since $y$ is in the support of $X\cdot  \ell_y$, we get 
\[1+\sum m(x_j)\leq d,\]
which shows $m(x_j)=1$ for all $i$. 

Let $\psi\colon X\dashrightarrow  \mathbf{G}(1,n+1)$ be the rational map that associates to a generic $y\in X$ the line $\ell_y$, and denote by $\Gamma_\psi\subset X\times \mathbf{G}(1,n+1)$ the graph of $\psi$. Consider the incidence correspondence
$$I=\{(x,\ell)\colon x\in \ell\}\, \subseteq \, X\times \mathbf{G}(1,n+1).$$
We can assume that $X$ does not contain any lines, and hence the projection $I\rightarrow \mathbf{G}(1,n+1)$ is  finite. Consider the cycle
\[A \ =_{\defi} \ \pro_{X\times X, *} \left( \pro_{\mathbf{G}(1,n+1)\times X}^*\Gamma_\psi^t \cdot \pro_{X\times \mathbf{G}(1,n+1)}^*I\, \right)\] on $X \times X$. 
So the support of $A$ is the set $\{ (x, y) \mid x \in \ell_y\}$.  The   image  $\ol{Z}$ of   $Z$ in $X \times X$ and possibly the diagonal $\Delta$ are among the irreducible components of this cycle, and denoting by $R$ the remaining components, we have
\[  A \ = \ \ol{Z} \, + \, m \Delta \, + \,  R. \]
By construction, $R$ dominates the second factor, and we assert that it cannot dominate the first. Indeed, were $R$ to dominate both factors, it would define a correspondence violating the degree bounds in~\eqref{lowerbounddelta}.

Next, observe that $A$ acts as the composition
$$
[A]^* \, =\, [I]^*\circ\psi_*\colon H^{n,0}(X)\longrightarrow H^{\bullet}(\mathbf{G}(1,n+1))\longrightarrow H^{n,0}(X),
$$
and this composition is zero since $H^{\bullet}(\mathbf{G}(1,n+1))$ is Hodge--Tate.   Furthermore,
$$
[R]^*=0\colon H^{n,0}(X)\longrightarrow H^{n,0}(X)
$$
since $R$ does not dominate the first factor. 
Therefore  $m=-c$, and in particular  $c\leq 0$. 
If $c\neq 0$, we contend that $c=-1$. Indeed, given a general point $(x,y)\in Z$,  the lines $_x\ell$ and $\ell_y$  pass through $x$ and $y$ and thus
$$_x\ell=\ell_y.$$
Consequently, $_x\ell\cdot X=\ell_y\cdot X$ and by the statements above, we see that $x$ and $y$ both appear with multiplicity $1$ in this intersection. Accordingly, the diagonal must appear with multiplicity $1$ in $ A$, and thus $c=-1$.

To finish the proof, we need the following.

\begin{claim}\label{claim}
Every irreducible component of $($the support of\;$)$ $R$ is of the form $x_0\times X$ for some $x_0\in X$.
\end{claim}

\begin{proof}
The proof proceeds exactly as the proof of \cite[Theorem A]{MOA}. In brief, if the projection of an irreducible component of $R$ to the first factor is $S$, one  shows that sections of the canonical bundle of a desingularization of $S$ do not birationally separate many points. This contradicts computations of Ein \cite{Ein} and Voisin \cite{Voisin} if $\dim S>0$.
\end{proof}

The claim implies that $R$ must be irreducible and reduced since lines meeting $X$ in any fixed zero-dimensional subscheme of $X$ of length 2 do not meet a general point of $X$. It follows that $\delta_b\geq d-2$, and by symmetry that $\delta_a\geq d-2$.

Theorem~\ref{NewTheoremD'} also follows from this claim as follows. 

\begin{proof}[Proof of Theorem~\ref{NewTheoremD'}]
\eqref{D'i}~
If $\delta_b=d-2$, we must have $c=-1$ and $R=x_0\times X$ for some $x_0\in X$. Then we have the equality
\[\ol{Z} \ = \ \textnormal{closure}\left\{(x,y)\mid x\neq x_0, y\neq x, x_0, \text{ and } x\in \overline{x_0y}\right\}.\]
Indeed, every irreducible component of the right-hand side must dominate the second factor, and the degree of the projection of the right-hand side to the second factor is $d-2$.

\eqref{D'ii}~
If $\delta_b=d-1$ and $c=0$, we will show that \eqref{D'iia} is satisfied. There is a point $x_0\in X$ such that $R=x_0\times X$. Consider $(x,y)\in Z$ general, and let 
\[_x\ell\cap X=\{y_j \mid 0\leq j\leq d-1\}\quad\text{and}\quad\ell_y\cap X=\{x_j\mid 0\leq j\leq \delta_a\},\]
where $x=x_1$ and $y=y_1$, so that $(x_1,y),\ldots, (x_{d-1},y)$ and  $(x,y_1),\ldots, (x,y_{\delta_a})$ are in $Z$. Since $(x,y)\in Z$ was chosen generically, $b^{-1}(y_2)$ consists of $d-1$ points, one of which is $x$. Moreover, $b^{-1}(y_2)$ is contained in a line passing through $x_0$, and thus is contained in the line through $x_0$ and $x$. It follows that
$$b^{-1}(y_2)=\{(x_i,y_2)\mid 1\leq i\leq d-1\}.$$
The same reasoning shows that
$$\{(x_i,y_j)\mid 1\leq i\leq d-1, 1\leq j\leq \delta_a\}\subset Z.$$
Let $\varphi_{1}\colon X\dashrightarrow \mathbf{P}^n$ be the projection from $x_0$, and consider the map
\begin{align*}\varphi_{2}\colon X&\longdashrightarrow \mathbf{G}(1,n+1)\\
x\ &\longmapsto \ \ \ \ \  {_x\ell} .\end{align*}
The maps $\varphi_1$ and $\varphi_2$ are generically finite of degree $d-1$ and at least $\delta_a$, respectively. Considering degrees in the following diagram, we see that $\varphi_2$ had degree $\delta_a$ and that $(\varphi_1\times\varphi_2)(\ol{Z})\subset \mathbf{P}^n\times \Im(\varphi_2)$ maps birationally to each factor: 
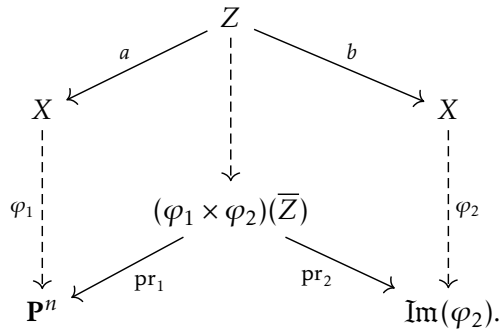
\[
\begin{tikzcd}
&Z \ar[dl,swap,"a"] \ar[dr,"b"] \ar[dd,dashed]&\\
X \ar[dd,swap,dashed,"\varphi_1"] &  & X \ar[dd,dashed,"\varphi_2"] \\
&(\varphi_1\times\varphi_2)(\ol{Z}) \ar[dl,"\pro_1"] \ar[dr,swap,"\pro_2"]&\\
\mathbf{P}^n&& \Im(\varphi_2)\rlap{.}
\end{tikzcd}
\]
Hence, the subvariety
\[(\varphi_1\times\varphi_2)(Z)\subset \mathbf{P}^n\times \Im(\varphi_2)
\] is the graph of a birational isomorphism $\psi\colon \mathbf{P}^n\dashrightarrow \Im(\varphi_2)$. Accordingly, $\ol{Z}$ is the fiber product of $\varphi_1$ and $\psi^{-1}\circ\varphi_2$.

Finally, if $\delta_b=d-1$ and $c\neq 0$, we show that \eqref{D'iib} is satisfied. We must have $c=-1$ and $\deg(a)=d-1$. Consider the rational map
\begin{align*}\varphi\colon X&\longdashrightarrow \mathbf{G}(1,n+1)\\
y \ &\longmapsto \ \ \ \ \ \ \ell_y.\end{align*}
Denoting by $U$ an open on which $\varphi$ is defined, we contend that 
\[\ol{Z}\, =\, \overline{\{(x,y)\in U^2\colon x\neq y,  \varphi(x)=\varphi(y)\}}\subset X\times X.\]
Given a generic $(x,y)\in Z$, the line $\ell_y$ coincides with the line $_x\ell$ as they both pass through $x$ and $y$. Write
\[_x\ell\cap X=\ell_y\cap X=\{z_j\colon 1\leq j\leq d\},\]
where $z_1=x$ and $z_2=y$. For any $j>1$, the point $(x,z_j)$ is on $Z$, and $b^{-1}(z_j)$ is contained in $\ell_{z_j}={_x\ell}=\ell_y$, so that
\[b^{-1}(z_j)=\ell_y\cap X\setminus \{z_j\}\]
and
\[\{(z_i,z_j)\colon i\neq j\}\subset Z.\]
It follows that $\ell_x= {_x\ell}$ for a generic $x\in X$ and that
\begin{equation*}\pushQED{\qed}
\ol{Z}\, =\, \overline{\{(x,y)\in U^2\colon x\neq y,  \varphi(x)=\varphi(y),\}}\subset X\times X.
\qedhere \popQED
	\end{equation*}
\renewcommand{\qed}{}     
\end{proof}

\section{Proof of Theorem~\ref{NewTheoremC}}\label{sec3}

Theorem~\ref{NewTheoremC} from the introduction follows easily from the following result. 

\begin{proposition}\label{PropositionE}
  Let $X$ be a very general hyperelliptic curve of genus $g\geq 3$, and let $Z\subseteq X\times X$ be a hyperelliptic curve. Then the image of\, $($the normalization of\;$)$ $Z$ under the Abel--Jacobi map is geometrically degenerate;  \textit{i.e.}, it generates a proper subtorus of\, $J(X)\times J(X)$.
\end{proposition}

Note that we do not assume that $Z$ is smooth; to say it is hyperelliptic means that its normalization is so. 

We note that some genericity condition is necessary in Theorem~\ref{NewTheoremC}. For example, given a hyperelliptic curve $X$, the graph of an automomorphism $X\rightarrow X$ which is neither the identity nor the hyperelliptic involution is a hyperelliptic curve  sitting in $X\times X$. The fact that such graphs map to geometrically degenerate curves in $J(X)\times J(X)$, together with Proposition~\ref{PropositionE}, suggests the following.

\begin{question}
Given an arbitrary hyperelliptic curve $X$ (resp.\ hyperelliptic curves $X$ and $Y$), does every hyperelliptic curve $Z\subseteq X\times X$ (resp.\ $Z\subset X\times Y$) map to a geometrically degenerate curve in $J(X)\times J(X)$ (resp.\ $J(X)\times J(Y)$)?
\end{question}

Let us first show how Theorem~\ref{NewTheoremC} follows from Proposition~\ref{PropositionE}. 

Consider a very general hyperelliptic curve $X$ and a hyperelliptic curve $Z\subset X\times X$ with normalization $Z'$. Abusing notation, we will call the image in $Z$ of Weierstrass points of $Z'$ Weierstrass points of $Z$. Such points map to Weierstrass points of $X$ under each projection. Consider a Weierstrass point $(x_0,y_0)\in Z$ and the embedding
\begin{align*}X\times X \ &\longrightarrow  \ \ \ J(X)\times J(X)\\
(x,y) \ \  \ &\longmapsto  \  ([x]-[x_0],[y]-[x_0]).\end{align*}
By Proposition~\ref{PropositionE}, a translate of the image of $Z$ in $J(X)\times J(X)$ is contained in an abelian subvariety of $J(X)\times J(X)$. Since the image of $Z$ passes through 
\[\tau\ =_{\defi}\ (0,[y_0]-[x_0])\in J(X)[2]\times J(X)[2],\] it is in fact contained in $\tau+A$ for some proper abelian subvariety $A\subset J(X)\times J(X)$. Moreover, since $X$ is very general, the automorphism group of the Jacobian of $X$ is $\ZZ$, and thus there are integers $m,n\in \ZZ$, $m\geq 0$, such that $Z$ is contained in the image of 
\begin{align*} J(X)&\lra J(X)\times J(X)\\
x \ \ & \longmapsto  \  (mx,nx)+\tau.
\end{align*}
Hence,
\[Z\subset \{(x,x')\in X\times X\colon nx+x_0=mx'+y_0\in J(X) \}\subset X\times X.\]
Equivalently, $Z$ is contained in the fiber of the following map over $y_0-x_0$: 
\begin{align*}X\times X&\longrightarrow \ \  J(X)\\
(x,x')&\longmapsto mx-nx'.\end{align*}
Considering the differential of the map above and the fact that the Gauss map of $X$ embedded in its Jacobian has degree $2$, it is easy to see that the only possibility is $n=\pm m$ and $x_0=y_0$.

We have thus shown that $Z$ is contained either in the diagonal of $J(X)$ or in the anti-diagonal of $J(X)$. This completes the proof as the diagonal of $J(X)$ intersects $X\times X$ along the diagonal of $X$ and the anti-diagonal of $J(X)$ intersects $X\times X$ along the graph of the hyperelliptic involution of $X$.

Finally, we give the proof of Proposition~\ref{PropositionE}.

\begin{proof}[Proof of Proposition~\ref{PropositionE}]
Consider $\mathcal{X}/S$, a locally complete family of hyperelliptic curves of genus $g$, and 
\[\mathcal{Z}\subset J(\mathcal{X}/S)\times_SJ(\mathcal{X}/S),\]
a family of hyperelliptic curves such that for very general $s\in S$, the curve $\mathcal{Z}_s$ generates $J(\mathcal{X}_s)\times J(\mathcal{X}_s)$. The idea is to arrive at a contradiction to the observation of Pirola \cite{Pirola} that hyperelliptic curves on abelian varieties are rigid up to translation. 

Specifically, specialize to loci $S_\lambda\subset S$ along which $J(\mathcal{X}_s)$ is isogenous to $\mathcal{A}_s^\lambda\times E$, where $E$ is a fixed elliptic curve and $\mathcal{A}^\lambda\rightarrow S_\lambda$ is a family of abelian $(g-1)$-folds. For each $\lambda$, we have a map
\[p_\lambda\colon \mathcal{Z}_s\longrightarrow E\times E\]
which is the composition of the inclusion of $\mathcal{Z}_s$ in $J(\mathcal{X}_s)\times J(\mathcal{X}_s)$ with the isogeny and the projection to the $E\times E$ factor.

\begin{claim}
The image of $\mathcal{Z}_s$ in $E\times E$ varies with $s\in S_\lambda$. 
\end{claim} 

But as we noted, this is impossible thanks to \cite{Pirola}, completing the proof. 

The claim is established along the lines of  \cite{Voisin.Abel.Var} and \cite{Martin}. Denoting by $\mathcal{G}/S$ the relative Grassmanian of $(g-1)$-planes in $T_{J(\mathcal{X}_s),0}$, \cite{CP} proves the density of the set 
\[\{T_{\mathcal{A}_s^\lambda,0}\, \subseteq \, T_{J(\mathcal{X}_s),0}\mid  s\in S_\lambda\}\ \subseteq \ \mathcal{G}.\]
(In fact, \cite{CP} shows that the locus 
$\{T_{E,0}\subset T_{J(\mathcal{X}_s),0}  \mid s\in S_\lambda\}$ is dense in the relative Grassmanian of lines in $T_{J(\mathcal{X}_s),0}$. However, one can use the fact that Jacobians are isomorphic to their duals to get the stated assertion.)
By a density argument, one can construct families of smooth curves $\widetilde{\mathcal{Z}}\rightarrow \mathcal{G}'$ and $\widetilde{\mathcal{Z}}'\rightarrow \mathcal{G}'$
 over a generically finite cover $\mathcal{G}'$ of $\mathcal{G}$ and a morphism
 \[p\colon \widetilde{\mathcal{Z}}\longrightarrow \widetilde{\mathcal{Z}}'\]
satisfying the following:
\begin{itemize}
\item Denoting by $\pi$ the map $\mathcal{G}'\rightarrow S$, the curve $\widetilde{\mathcal{Z}}_s$ is the normalization of $\mathcal{Z}_{\pi(s)}$.
\item For $s\in \mathcal{G}'$ such that $\pi(s)\in S_\lambda\subset S$, the maps
\[p_\lambda\colon \mathcal{Z}_s\longrightarrow p_\lambda(\mathcal{Z}_s)\]
and
 \[p\colon \widetilde{\mathcal{Z}}_s\longrightarrow \widetilde{\mathcal{Z}}_s'\]
 agree birationally.
 \end{itemize}
Now consider the composition
 \begin{equation}\label{picmap1}
 \Pic^0(J(\mathcal{X}_s)\times J(\mathcal{X}_s))\longrightarrow \Pic^0(\widetilde{\mathcal{Z}}_s)\xrightarrow{p_*}\Pic^0(\widetilde{\mathcal{Z}}_s'),\end{equation}
 where the first map is the pullback by the composition
 \[\widetilde{\mathcal{Z}}_s\longrightarrow \mathcal{Z}_s\longrightarrow J(\mathcal{X}_s)\times J(\mathcal{X}_s).\]
 One easily checks that the composition \eqref{picmap1} cannot be zero. Since $J(\mathcal{X}_s)$ is simple for generic $s\in \mathcal{G}'$, we deduce that the abelian variety $\Pic^0(\widetilde{\mathcal{Z}}_s')$ contains an abelian subvariety isogenous to $J(\mathcal{X}_s)$ for all $s$ in an open set $U\subset \mathcal{G}'$.
 
Finally, consider $\lambda$ such that $\pi^{-1}(S_\lambda)\cap U\neq\emptyset$. If $p_\lambda(\mathcal{Z}_s)$ does not vary with $s\in S_\lambda$, the fixed abelian variety $\Pic^0(\widetilde{\mathcal{Z}}_s')$ contains an abelian subvariety isogenous to $J(\mathcal{X}_s)$ for all $s\in S_\lambda$. This cannot be since the family $J(\mathcal{X}_{S_\lambda}/S_\lambda)$ is not isotrivial.

\end{proof}


\newcommand{\etalchar}[1]{$^{#1}$}

 \end{document}